\documentclass{article}
\usepackage{arxiv}

\usepackage[utf8]{inputenc} 
\usepackage[T1]{fontenc}    
\usepackage{hyperref}       
\usepackage{amsfonts}       
\usepackage{nicefrac}       
\usepackage{mathtools}
\usepackage{amsthm}
\usepackage{physics}
\usepackage{wasysym}
\usepackage{bbm}

\newtheorem{theorem}{Theorem}
\newtheorem{lemma}{Lemma}
\newtheorem{proposition}{Proposition}
\newtheorem{corollary}{Corollary}
\theoremstyle{definition}

\newtheorem{assumption}{Assumption}

\newtheorem{remark}{Remark}

\DeclareMathOperator{\loc}{loc}
\DeclareMathOperator{\adj}{adj}
\newcommand{\mi}{\mathrm{i}}
\renewcommand{\rank}{\operatorname{rank}}
\renewcommand{\arg}{\operatorname{arg}}
\DeclareMathOperator{\minn}{min}
\DeclareMathOperator{\maxx}{max}
\renewcommand{\Re}{\operatorname{Re}}
\renewcommand{\Im}{\operatorname{Im}}
\newcommand{\pmult}{.\hspace{-0.1cm}*}

\usepackage{hyperref}
\usepackage{mathtools}
\usepackage{tcolorbox}
\usepackage{amsmath}
\usepackage{nicefrac}

\title{Runge-Kutta convolution quadrature \\ based on Gauss methods}

\author{{\small	Lehel Banjai} \\
{\small 	Maxwell Institute for Mathematical Sciences} \\
{\small 	School of Mathematical \& Computer Sciences}\\ 
{\small	Heriot-Watt University} \\
{\small	Edinburgh, EH14 4AS, United Kingdom} \\
{\small	\texttt{l.banjai@hw.ac.uk}}
\And
{\small	Matteo Ferrari} \\
{\small	Dipartimento di Scienze Matematiche ``G.L. Lagrange'' }\\ 
{\small	Politecnico di Torino} \\
{\small	Torino, 10129, Italy} \\
{\small	\texttt{matteo.ferrari@polito.it}}
}

\renewcommand{\headeright}{}
\renewcommand{\undertitle}{}
\renewcommand{\shorttitle}{}

\hypersetup{
pdftitle={Gauss_CQRK},
pdfsubject={q-bio.NC, q-bio.QM},
pdfauthor={Ferrari M.~Banjai L.},
pdfkeywords={Runge-Kutta, Gauss methods, Convolution quadrature},
}

\begin{document}
\maketitle

\begin{abstract}
An error analysis of Runge-Kutta convolution quadrature based on Gauss methods applied to hyperbolic operators is given. The order of convergence relies heavily on the parity of the number of stages, a more favourable situation arising for the odd cases than the even ones. Moreover, for particular situations the order of convergence is higher than for Radau IIA or Lobatto IIIC methods when using the same number of stages. We further investigate an application to transient acoustic scattering where, for certain scattering obstacles, the favourable situation occurs in the important case of the exterior Dirichlet-to-Neumann map. Numerical experiments and comparisons show the performance of the method.
\end{abstract}


\section{Introduction}

In this paper, we develop and analyze convolution quadrature  based on Runge-Kutta Gauss methods used to compute integrals of the form
\begin{equation} \label{first_eq}
	\int_0^t k(t-\tau) g(\tau) \, \text{d} \tau, \quad t > 0
\end{equation}
for a given kernel $k$ and a casual function $g$. The kernel $k$ is not given directly, but instead via its Laplace transform $K$ which is assumed to be of hyperbolic type, i.e., $K$ is assumed to be analytic and polynomially bounded in a region containing the half-space $\Re s \ge \sigma_0 > 0$. Convolutions of this type arise, for example, when solving partial differential equations of hyperbolic or parabolic type by boundary integral equations \cite{CostabelSayas2017, BanjaiSayas2022}. The convolution quadrature method was developed in \cite{Lubich1988a,Lubich1988b,LubichOstermann1993} for parabolic  and in \cite{Lubich1994} for hyperbolic problems. The underlying idea is to replace the convolution kernel $k$ in \eqref{first_eq} with the inverse Laplace transform of $K$, thus arriving at
\begin{equation} \label{seconddd_eq}
	\int_0^t k(t-\tau) g(\tau) \, \text{d} \tau = \frac{1}{2\pi\mi}\int_{\mathcal{C}} K(s) y_s(t) \text{d}s, \quad t > 0
\end{equation}
where $\mathcal{C}$ is an appropriate contour and $y_s$ solves the ordinary differential equation (ODE)
\begin{equation*}
y_s'(t) = sy_s(t) +g(t), \qquad y_s(0) = 0.
\end{equation*}
Convolution quadrature is obtained by replacing $y_s$ with its approximation obtained by solving the ODE by an appropriate linear multistep or Runge-Kutta method.

For multistep-based convolution quadrature, $A$-stability of the underlying ODE-solver is necessary, therefore the achievable convergence order is limited by the Dahlquist barrier $p = 2$. On the other hand, $A$-stable Runge-Kutta methods of arbitrary order are available and convolution quadratures based on them often outperform those based on linear multistep methods. In \cite{BanjaiLubich2011,BanjaiLubichMelenk2011} error analysis of Runge-Kutta convolution quadrature applied to hyperbolic kernels is developed. However, in both works the analysis is limited to stiffly accurate $A$-stable Runge-Kutta methods, among which we mention Radau IIA and Lobatto IIIC methods. In this paper we extend the theory of convolution quadrature for hyperbolic symbols to Runge-Kutta methods based on Gauss-Legendre quadratures nodes. The analysis is mainly based on the location of zeros related of the Pad{\'e} approximants of the exponential function;  these rational polynomials are the stability functions of the Runge-Kutta Gauss methods.

The main motivation for extending the convolution quadrature theory to Gauss methods is their energy conservation property. Our main result concerns their convergence order and its dependence on the number of stages of the Gauss Runge-Kutta method, and on the growth exponent of the Laplace transform $K$. Moreover, we show that the order of convergence relies heavily on the parity of the stage order, a better situation arising for the odd cases than the even ones. On the other hand, a weaker regularity requirement for the datum $g$ is needed for Gauss methods with even order of stages.  For particular kernels, when using Gauss methods with odd number of stages, the order of convergence is higher than for Radau IIA methods with the same number of stages. This situation includes the important case of the exterior Dirichlet-to-Neumann map and the inverse of the single-layer operator for the acoustic scattering problem with a  convex scattering obstacle.

The paper is organized as follows: in the next section we define hyperbolic symbols, and recall some properties of the Runge-Kutta methods and associated convolution quadrature formulas.
In Section \ref{sec_gauss} we analyze the stability function of the Runge-Kutta Gauss methods, and we present some results on the zeros of associated complex polynomials.
In Section \ref{sec_Main} we develop the error analysis, the main result of the paper is proved, and  numerical tests are shown that support the results of the theory.
Finally, in Section \ref{sec_wave} as an application of our main theorem, we also study Gauss Runge-Kutta convolution quadrature for a boundary integral equation formulation of a wave scattering problem, and we present two numerical tests which confirm the sharpness of the theoretical results.

\section{Setting} \label{sec_sett}
\subsection{Hyperbolic symbols}
We denote the space of bounded linear operators between Banach spaces $X$ and $Y$ by $\mathcal{B}(X,Y)$ and denote by $\mathbb{C}_+ = \{s \in \mathbb{C} : \Re s > 0\}$ the right complex half-plane. We assume that $K(s)$, the Laplace transform of a kernel $k(t)$, is a hyperbolic symbol, i.e., that it satisfies the following assumption.
\begin{assumption} \label{assumption1}
The function $K(s) : \mathbb{C}_+ \to \mathcal{B}(X,Y)$ is analytic in the half-plane $\Re s \ge \sigma_0 > 0$, and bounded there as
\begin{equation*} \label{eq1}
	\| K(s) \|_{\mathcal{B}(X,Y)} \le M \abs{s}^\mu
\end{equation*}
for some $\mu \in \mathbb{R}$ and a constant $M>0$.
\end{assumption}
\begin{remark} \label{remark1}
Often hyperbolic symbols satisfy a more refined bound of the form
\[
	\| K(s) \|_{\mathcal{B}(X,Y)} \le M \frac{\abs{s}^\mu}{(\Re s)^\nu}, \quad \text{for all~} \Re s \ge \sigma_0 > 0
\]
for some real exponent $\nu \ge 0$. However, in contrast to  the analysis in \cite{BanjaiLubichMelenk2011}, we are not able to make use of the  resulting stronger  bound in convex sectors of the form $\vert \arg(s) \vert < \nicefrac{\pi}{2} - \theta$ for $\theta \in (0,\pi/2)$. This is due to Gauss methods not being $L$-stable, i.e., their stability region is exactly the right-half complex plane. An effect of this can be seen in the existence of poles of the discretized integrands in \eqref{seconddd_eq} localized in the vicinity of the imaginary axis.
\end{remark}
If $\mu < -1$, the inverse Laplace transform $k = \mathcal{L}^{-1} K$ is a causal ($k(t) = 0$, $t < 0$), continuous function and thus for integrable $g$, we can define the convolution
\begin{equation*}
	u(t) = K(\partial_t) g(t) = \int_0^t k(t-\tau) g(\tau) \dd\tau, \quad t>0.
\end{equation*}
For a general $\mu \in \mathbb{R}$, we define the convolution via 
\begin{equation} \label{eq2p1}
	u(t) = K(\partial_t)g(t) =  \mathcal{L}^{-1} \{ K(s) G(s)\}(t), \quad t > 0,
\end{equation}
where $G$ is the Laplace transform of $g$. This is consistent with the above definition of the convolution in the case $\mu < -1$. Similarly, the case  $\mu > -1$ can be investigated by assuming that for $m > \mu+1$, $g \in C^{m-1}(\mathbb{R})$ is a causal function satisfying $g^{(j)}(0) =0 $, $j = 0,\dots,m-1$, and $g^{(m)}$ is locally integrable. Then,
\begin{equation*}
	K(\partial_t) g (t) = \frac{d^m}{dt^m} \int_0^t k_m(t-\tau) g(\tau) \dd\tau, \quad t > 0,
\end{equation*}
where $k_m$ is the inverse Laplace transform of $K_m(s) = s^{-m} K(s)$.

The motivation behind  the operational notation  $K(\partial_t)g$, can be seen when considering the case $K(s) = s^m$, where the above definition implies that $K(\partial_t)g = \partial_t^m g$, where $\partial_t$ denotes the causal derivative (see e.g. \cite{BanjaiSayas2022}). Furthermore, the composition rule $K_2K_1(\partial_t)g = K_2(\partial_t)K_1(\partial_t)g$ holds for hyperbolic kernels $K_1$ and $K_2$. In the next section we describe convolution quadrature, a numerical method for approximating $K(\partial_t) g$ that conserves such properties of the operational convolution.

%
\subsection{Runge-Kutta convolution quadrature}
We consider an implicit Runge-Kutta method with $m$ stages, and coefficients $\{b_i, a_{ij}, c_i\}$. The discretization with time step $h>0$ of the initial value problem $y'(t) = f(t,y(t)), \, y(0) = y_0$ is given by the recurrence
\begin{align*}
	Y_{ni} & = y_n + h \sum_{j=1}^m a_{ij} f(t_n + c_j h,Y_{nj}), \quad  i = 1,\ldots,m \, , 
	\\ y_{n+1} & = y_n + h \sum_{j=1}^m b_{j} f(t_n + c_j h,Y_{nj}),
\end{align*} 
where $t_j = j h$. The Runge-Kutta method has (classical) order $p$ and stage order $q \le p$, if for sufficiently smooth $f$,
\begin{equation*}
	y_1 = y(h)+\mathcal{O}(h^{p+1}) \quad \text{~and~} \quad Y_{0j} = y(t_0+c_jh) + \mathcal{O}(h^{q+1}),
\end{equation*}
respectively (see \cite{HairerWanner1991}). We make use of the standard notation based on the Butcher tableau: 
\begin{align*} 
	& b^T = (b_1,\ldots,b_m), \qquad \qquad \quad c = (c_1,\ldots,c_m)^T, 
	\\ &A = 
	\begin{pmatrix} 
		a_{11} & \dots & a_{1m} \\ \vdots && \vdots \\ a_{m1} & & a_{mm} 
	\end{pmatrix}, 
	\hspace{0.1cm} \quad \mathbbm{1} = (1,\ldots,1)^T \in \mathbb{R}^m.
\end{align*}

A Runge-Kutta method is said to be \textit{A-stable} if $I-zA$ is non singular for $\Re z \le 0$ and the stability function
\begin{equation*}
	R(z) = 1 + z b^T (I - zA)^{-1} \mathbbm{1},
\end{equation*}
satisfies $\abs{R(z)} \le 1$ for $\Re z \le 0$. If $A^{-1}$ exists, we denote $R(\infty) = \lim_{\abs{z} \to \infty} R(z) = 1 - b^T A^{-1} \mathbbm{1}$. Important families of Runge-Kutta methods that are $A$-stable are Radau IIA, Lobatto IIIC and Gauss methods.

We describe a CQ discretization of \eqref{eq2p1} based on the above Runge-Kutta method, as  first introduced in \cite{LubichOstermann1993} for parabolic operators, and then analyzed for hyperbolic operators in \cite{BanjaiLubich2011, BanjaiLubichMelenk2011}. We refer also to \cite{CalvoCuestaPalencia2007} for an application and analysis of Runge-Kutta CQ to homogeneous Volterra equations, which includes Gauss methods.

Considering a uniform time step $h$, at the time $t_n = h n$, the discretization is given by
\begin{equation*}
	K(\partial_t^h)g(t_n) = u_n = \sum_{j=0}^n \gamma_{n-j} U_j \qquad \text{~and~} \qquad U_n = \sum_{j=0}^n W_{n-j}(K) g(t_n + c h),
\end{equation*}
where the coefficients are determined by generating functions
\begin{equation*}
	K \left( \frac{\Delta(\zeta)}{h} \right) = \sum_{j=0}^\infty  W_j(K) \zeta^j, \quad\gamma(\zeta) = \sum_{j=0}^\infty  \gamma_j \zeta^j,
\end{equation*}
and
\begin{equation*}
	\Delta(\zeta) = \left( \frac{\zeta}{1-\zeta} \mathbbm{1} b^T  + A\right)^{-1}, \quad \gamma(\zeta) = \frac{\zeta b^T A^{-1}}{1-\zeta(1-\beta^TA^{-1} \mathbbm{1})}.
\end{equation*}
The spectrum of $\Delta(\zeta)$ for $\abs{\zeta}<1$ plays a key role in the analysis of Runge-Kutta convolution quadrature. We recall that in \cite[Lemma 2.6]{BanjaiLubichMelenk2011} it has been proved that if $A$ is invertible and the Runge-Kutta method is $A$-stable, then for $\abs{\zeta}<1$
\begin{equation}\label{BLM}
	\sigma\left(\Delta(\zeta)\right) \subset \sigma(A^{-1}) \cup \{ z \in \mathbb{C} : R(z)\zeta = 1\}.
\end{equation}
We are able to prove a slightly improved version of this result by virtue of the following auxiliary lemma.
\begin{lemma} \label{butcherlemma}
An implicit Runge-Kutta method whose coefficient matrix $A$ is invertible, $b_j \ne 0$ for all $j$, and $c_{i_1} \ne c_{i_2}$ for all $i_1 \ne i_2$, satisfies 
\begin{equation} \label{first_eig}
	b^T x \ne 0 \text{ for all eigenvectors~} x \text{ of~} A,
\end{equation}
and
\begin{equation} \label{second_eig}
	\mathbbm{1}^T y \ne 0 \text{ for all eigenvectors~} y \text{ of } A^T.
\end{equation}
\end{lemma}
\begin{proof}
The proof relies on the following properties of implicit Runge-Kutta methods (see \cite{Butcher1964}): for all $k = 1, \ldots, m$, where $m$ is the number of stages, it holds
\begin{equation} \label{butcher1}
	\sum_{j=1}^m b_j c_j^{k-1} a_{ij} = \frac{b_i(1-c_i^k)}{k}, \qquad \text{for~} i= 1, \ldots, m, 
\end{equation}
and
\begin{equation} \label{butcher2}
	\sum_{j=1}^m c_j^{k-1} a_{ij} = \frac{c_i^k}{k}, \qquad \text{for~} i= 1, \ldots, m. 
\end{equation}
Let us start by proving property \eqref{first_eig}. By contradiction suppose $b^T x =0$ with $x = (x_1,\ldots,x_m)^T \in \mathbb{R}^m \setminus \{0\}$ such that $Ax = \lambda x$ for some $\lambda \ne 0$. 
Then, we obtain, using \eqref{butcher1} with $k=1$,
\begin{equation*}
	0 = \lambda b^T x = b^T A x = (b \pmult (\mathbbm{1}-c))^T x = b^T x - (b \pmult c)^T x = (b \pmult c)^T x,
\end{equation*}
where $.*$ stands for the pointwise multiplication between vectors.
We iterate the procedure, using \eqref{butcher1} with $k=2$,
\begin{equation*}
	0 = \lambda (b\pmult c)^T x = (b\pmult c)^T A x = \frac{(b \pmult (\mathbbm{1}-c^2))^T}{2} x = \frac{b^T x}{2} - \frac{(b \pmult c^2)^T x}{2} = \frac{(b \pmult c^2)^T x}{2}.
\end{equation*}
Using the same mechanism we can show that $(b \pmult c^k)^T x = 0$ for $k=0,\ldots,m$. The first $m$ equations are equivalent to the system
\begin{equation*}
	\begin{pmatrix}
		b_1 & b_2 & \cdots & b_m \\
		b_1c_1 & b_2c_2 & \cdots & b_mc_m \\
		b_1c_1^2 & b_2c_2^2 & \cdots & b_mc_m^2 \\
		\cdots & \cdots & \cdots & \cdots \\
		b_1c_1^{m-1} & b_2c_2^{m-1} & \cdots & b_mc_m^{m-1} 
	\end{pmatrix}
	\begin{pmatrix}
		x_1 \\
		x_2 \\
		\cdots \\
		\cdots \\
		x_m
	\end{pmatrix}
	=
	\begin{pmatrix}
		0 \\
		0 \\
		\cdots \\
		\cdots \\
		0
	\end{pmatrix}
\end{equation*}
of which the only solution is $x = 0$, and this leads to a contradiction. Indeed, the matrix on the left hand side is essentially a Vandermonde matrix whose determinant is
\begin{equation*}
	\begin{vmatrix}
		b_1 & b_2 & \cdots & b_m \\
		b_1c_1 & b_2c_2 & \cdots & b_mc_m \\
		b_1c_1^2 & b_2c_2^2 & \cdots & b_mc_m^2 \\
		\cdots & \cdots & \cdots & \cdots \\
		b_1c_1^{m-1} & b_2c_2^{m-1} & \cdots & b_mc_m^{m-1} 
		\end{vmatrix} 
		= \prod_{j=1}^m b_j \prod_{1 \le i_1 < i_2 \le m} (c_{i_1} - c_{i_2}) 
\end{equation*}
and this is not zero since all the $b_j$ are non zero and the $c_i$ are different. 

The proof of \eqref{second_eig} is similar. We start by contradiction supposing that there exists $y \in \mathbb{R}^m \setminus \{0\}$ such that $\mathbbm{1}^T y = 0$ and $A^T y = \mu y$ for $\mu \ne 0$. First, using \eqref{butcher2} iteratively we can show that $(c^k)^T y = 0$ for $k=0,\ldots,m$. The first $m$ equations are again equivalent to a system whose matrix is now exactly a Vandermonde matrix associated to the vector $c$, whose determinant is not zero since all the $c_i$ are different. This leads to a contradiction.
\end{proof}
We are able to state a new result on the characterization of $\sigma(\Delta(\zeta))$ under the following hypothesis.
\begin{assumption}\label{assump2}
We assume that the Runge-Kutta coefficient matrix $A$ is invertible with all its eigenvalues of multiplicity one.
\end{assumption}
We have numerically verified that all the Radau IIA, Lobatto IIIC and Gauss methods satisfy Assumption \ref{assump2} up to $m=6$ stages. It is worth investigating if Assumption \ref{assump2} is always satisfied by these methods, but this goes beyond the scope of present paper.

\begin{proposition} \label{prop1}
Assume an $A$-stable Runge-Kutta method satisfies Assumption \ref{assump2}. Moreover, suppose that $b_j \ne 0$ for all $j$, and $c_{i_1} \ne c_{i_2}$ for all $i_1 \ne i_2$.
Then, the following holds 
\begin{equation} \label{result_prop1}
	\sigma(\Delta(\zeta)) = \{ z \in \mathbb{C} : R(z)\zeta = 1\}
\end{equation}
for all $0 < \abs{\zeta} < 1$.
\end{proposition}
\begin{proof}
Let $\lambda$ be an eigenvalue of $A^{-1}$. We want to show that $\det(\lambda I - \Delta(\zeta)) \ne 0$, for all $0 < \abs{\zeta} < 1$. If the latter holds, then, by virtue of \eqref{BLM}, we can deduce \eqref{result_prop1}.
Using the Sherman-Morrison formula we can rewrite
\begin{equation*}
	\Delta(\zeta) = A^{-1} - \zeta \frac{A^{-1} \mathbbm{1} b^T A^{-1}}{1- \zeta R(\infty)},
\end{equation*}
and thus
\begin{align*}
	\det(\lambda I - \Delta(\zeta)) = \det( \lambda I -  A^{-1} + \zeta \frac{A^{-1} \mathbbm{1} b^T A^{-1}}{1- \zeta R(\infty)}) = \det( A \lambda -  I + \zeta \frac{A^{-1} \mathbbm{1} b^T }{1- \zeta R(\infty)}) \det\bigl(A^{-1}\bigr).
\end{align*}
To evaluate the determinant, we recall \cite[Equation (9)]{Marcus1990}: for all $B \in \mathbb{C}^{n,n}, u \in \mathbb{C}^n, v \in \mathbb{C}^n$ it holds that
\begin{equation} \label{eq18}
	\det\bigl(B+uv^T\bigr) = \det(B) + v^T \adj(B) u
\end{equation}
where $\adj$ stands for the adjugate matrix.
We apply \eqref{eq18} with $B = A \lambda - I$, $u = A^{-1}\mathbbm{1}$, $v = \frac{\zeta}{1-\zeta R(\infty)} b$ to obtain 
\begin{align*}
	\det(\lambda I - \Delta(\zeta)) & = \det(A\lambda-I) +  \frac{\zeta}{1-\zeta R(\infty)} b^T \adj(A \lambda - I) A^{-1} \mathbbm{1} \det\bigl(A^{-1}\bigr) 
	\\  &=  \frac{\zeta}{1-\zeta R(\infty)} b^T \adj(A \lambda - I) A^{-1} \mathbbm{1} \det\bigl(A^{-1}\bigr),
\end{align*}
where in the last step we used that $\lambda$ is an eigenvalue of $A^{-1}$. It remains to show that 
\begin{equation} \label{eq21}
	b^T \text{adj}(A\lambda -I)A^{-1} \mathbbm{1} \ne 0 \quad \text{~for all~} \lambda \text{~eigenvalue of~} A^{-1}.
\end{equation}
By hypothesis all the eigenvalues of $A^{-1}$ are simple, so for any eigenvalue $\lambda$ of $A^{-1}$, $\rank(\adj(A\lambda-I))=1$. Therefore, for all $\lambda$ there exists $\mu \in \mathbb{R} \setminus\{0\}$ and $x,y \in \mathbb{R}^m \setminus\{0\}$ such that $\adj(A\lambda - I) = \mu x y^T$. Employing the well-known property $\adj(B)B =B \adj(B) = \det(B) I$ for all $B \in \mathbb{C}^{n,n}$, we deduce that $x$ is an eigenvector of $A$ and $y$ of $A^T$. Therefore, condition \eqref{eq21} is equivalent to
\begin{equation*}
b^T x \ne 0 \text{~and~} \mathbbm{1}^T y \ne 0 \text{ for all~eigenvectors } x \text{ of~} A \text{~and all eigenvectors~} y \text{ of } A^T,
\end{equation*}
but the latter is true by Lemma \ref{butcherlemma}.
\end{proof}

\section{Gauss methods} \label{sec_gauss}
The $m$-stage Runge-Kutta method based on Gauss-Legendre quadrature nodes has stage order $m$ and classical order $2m$. Its stability function $R_m(z)$ is the $(m,m)$-Pad{\' e} approximant of the exponential $e^z$. 
Namely, defining the monic polynomial
\begin{equation} \label{eq23}
	P_m(z) = \sum_{j=0}^m p_j z^j  \quad \text{with} \quad p_j = \frac{(2m-j)!}{j!(m-j)!},
\end{equation}
the stability function associated to the $m$-stage Gauss Runge-Kutta is
\begin{equation*}
	R_m(z) = \frac{P_m(z)}{P_m(-z)}.
\end{equation*}
By virtue of Proposition \ref{prop1}, we are interested in investigating the solutions of equations of the type $R_m(z) \zeta = 1$. 
\subsection*{Solutions of $R_m(z) = e^{t + \mi \theta}$ for $t \in \mathbb{C}$ and $\theta \in [-\pi,\pi]$, when $\abs{t} \to 0$}

Exploiting the expression of $R_m(z)$, we find that the $m$ solutions $\{\hat z_j(t,\theta)\}_{j=1}^m$ of
\begin{equation} \label{eq25}
	R_m(z) = e^{t + \mi \theta}
\end{equation}
 are the zeros of the polynomial (in $z$)
\begin{equation*}
	P_m(z) - e^{t+\mi\theta}P_m(-z) = \sum_{j=0}^m p_j \left(1+(-1)^{j+1} e^{t+\mi\theta} \right) z^j.
\end{equation*}
Using the expansion $e^{t} = 1 + t + \mathcal{O}(t^2)$, we readily obtain 
\begin{equation*} 
	P_m(z) - e^{t+\mi\theta}P_m(-z)  = M_m(\theta,z) - te^{\mi \theta} P_m(-z) + \mathcal{O}(t^2) P_m(z),
\end{equation*}
where 
\begin{equation} \label{eq27}
M_m(\theta,z) = P_m(z) - e^{\mi\theta}P_m(-z).
\end{equation}
We look for the solutions $\{ \hat z_j(t,\theta) \}_{j=1}^m$ in the vicinity of the zeros of the polynomial (in $z$) $M_m(\theta,z)$. A first, crucial, result is the following.
\begin{proposition} \label{prop2}
For all $\theta \in [-\pi,\pi]$, all the zeros (in $z$) of $M_m(\theta,z)$ are purely imaginary.
\end{proposition}
\begin{proof}
We note that all the zeros of $P_m(z)$ are in the open left half-plane (see e.g. \cite{SaffVarga1975}) and thus the polynomial $P_m(-z)$ has all its zeros in the open right-half complex plane. Let us suppose by contradiction that $M_m(\theta,z)$ has a zero $\hat z$ with $\Re \hat z > 0$. If we let $\{\hat s_j \}_{j=1}^m$ and $\{-\hat s_j \}_{j=1}^m$ be the zeros of $P_m(z)$ and $P_m(-z)$, respectively, from \eqref{eq27} we deduce that
$\abs{P_m(\hat z)} = \abs{P_m(-\hat z)}$, but this is impossibile since
\begin{equation*}
	\abs{P_m(\hat z)} = \prod_{j=1}^m \vert \hat z - \hat s_j \vert < \prod_{j=1}^m \abs{\hat z + \hat s_j} = \abs{P_m(-\hat z)}.
\end{equation*}
A similar contradiction is obtained when assuming $\Re \hat z  < 0$.
\end{proof}

Before characterizing the solutions of \eqref{eq25} we need the two following technical lemmas.

\begin{lemma} \label{lemma7}
Let $F(z)$ be analytic in a vicinity of the left half-plane and let
\begin{equation*}
\vert F(z) \vert \leq 1, \qquad \text{for } \Re z \leq 0.
\end{equation*}
Then, $F(\mi y) = e^{\mi \theta}$ for some $\theta, y \in \mathbb{R}$ implies $F'(\mi y) \neq 0$.
\end{lemma}
\begin{proof}
We will proceed by contradiction. Suppose $F'(\mi y) = 0$. As $F$ is analytic in the vicinity of $z = \mi y$ we have that
\begin{equation*}
F(z) = e^{\mi \theta}+c (z-\mi y)^k + \mathcal{O}\bigl((z-\mi y)^{k+1}\bigr),
\end{equation*}
for some integer $k \geq 2$ and some $c \in \mathbb{C}\setminus \{0\}$. Thus
\begin{equation*}
\vert F(z) \vert = \vert 1+re^{\mi \varphi} (z-\mi y)^k \vert+ \mathcal{O}((z-\mi y)^{k+1}),
\end{equation*}
for some $r > 0$, $\varphi \in (-\pi,\pi]$. Let us now choose $z$ so that for some $\varepsilon > 0$
\begin{equation*}
(z-\mi y)^k = \varepsilon^k e^{-\mi \varphi},
\end{equation*}
i.e., for any $j \in \mathbb{Z}$ we can set 
\begin{equation*}
z = \mi y+\varepsilon e^{-\mi \frac{(\varphi+2\pi j)}{k}}.
\end{equation*}
We want to choose $j$ so that 
\begin{equation*}
\Re z = \varepsilon \cos\left(\frac{\varphi+2\pi j}{k}\right) \leq 0.
\end{equation*}
Since $k \geq 2$, note that we necessarily have that $\nicefrac{\vert \varphi \vert}{k} \leq \nicefrac{\pi}{2}$. Let us first assume that $\varphi \geq 0$ and set $j = \lfloor \nicefrac{k}{2} \rfloor$. If $k$ even, then $j = \nicefrac{k}{2}$ and we have that
\begin{equation*}
\pi \leq \frac{\varphi+2\pi j}{k} = \frac{\varphi}{k}+\pi \leq \frac{3}{2}\pi.
\end{equation*}
If $k$ odd, then $j = \nicefrac{(k-1)}{2}$ and $k \geq 3$ and thus
\begin{equation*}
\frac{2}{3} \pi \leq \frac{\left(\varphi+2\pi j\right)}{k} = \frac{\varphi}{k}+\pi\left(1-\frac{1}{k}\right) < \frac{3}{2}\pi.
\end{equation*}
Therefore in both cases $\Re z \leq 0$ and
\begin{equation*}
\vert F(z) \vert = 1+r \varepsilon^k + \mathcal{O}\bigl(\varepsilon^{k+1}\bigr),
\end{equation*}
implying that for small enough $\varepsilon > 0$, $\vert F(z) \vert > 1$, which gives us the contradiction we were looking for. For $\varphi \leq 0$, the same contradiction is obtained with the choice $j = -\lfloor \nicefrac{k}{2} \rfloor$.
\end{proof}

\begin{lemma} \label{lemma8}
Let $R_m(z) = \nicefrac{P_m(z)}{P_m(-z)}$ be the $(m,m)$-Pad\'e approximant to $e^z$ and let $R_m(\mi y) = e^{\mi \theta}$ for some $\theta, y \in \mathbb{R}$. Then for $\abs{t} < r$ with small enough $r > 0$, there exists a unique  $z(t)$   in the vicinity of $\mi y$ such that $R_m(z(t)) = e^{t+\mi \theta}$. Furthermore
\begin{equation*}
z(t) = \mi y+ \beta t +\mathcal{O}(t^2),
\end{equation*}
when $\abs{t} \to 0$, where 
\begin{equation*}
\beta = \frac{1}{1-\frac{y^{2m}}{\vert P_m(\mi y) \vert^2}} > 1.
\end{equation*}
\end{lemma}
\begin{proof}
By virtue of Lemma \ref{lemma7} it holds $R'(\mi y) \neq 0$, therefore, the implicit function theorem implies that
\begin{equation*}
z(t) = \mi y+ \beta t +\mathcal{O}(t^2),
\end{equation*}
when $\abs{t} \to 0$, with
\begin{equation*}
\beta = \frac{e^{\mi \theta}}{R'(\mi y)}.
\end{equation*}
To prove the simplified expression for $\beta$ we will use the following properties of the polynomials $P_m$:
\begin{equation*}
2P_m'(z) = P_m(z)-zP_{m-1}(z)
\end{equation*}
(see \cite[Lemma 7]{Ehle1973}), and that the only term with an odd power of $z$ in the product $P_m(-z)P_{m-1}(z)$ is the term of highest power, i.e., $(-1)^m z^{2m-1}$ (see \cite[Lemma 1]{Ehle1973}). 

Using $R(\mi y) = \nicefrac{P_m(\mi y)}{P_m(-\mi y)} = e^{\mi \theta}$ we have that
\begin{equation*}
  \begin{split}    
e^{-\mi \theta}R'(\mi y) &= e^{-\mi \theta}\left(\frac{P_m'(\mi y)}{P_m(-\mi y)}+ \frac{P_m'(-\mi y)P_m(\mi y)}{P_m(-\mi y)^2}\right)= \frac{P_m'(\mi y)}{P_m(\mi y)}+\frac{P_m'(-\mi y)}{P_m(-\mi y)}
	\\ & = 2\Re \left(\frac{P_m'(\mi y)}{P_m(\mi y)}\right) = 2\Re \left(\frac{P_m'(\mi y)P_m(-\mi y)}{\vert P_m(\mi y) \vert^2}\right)
	\\ &= 1-\Re\left(\frac{\mi y P_m(-\mi y)P_{m-1}(\mi y)}{\vert P_m(\mi y)\vert^2}\right)= 1-\frac{y^{2m}}{\vert P_m(\mi y)\vert^2}.
  \end{split}
\]
Thus we have obtained the required expression for $\beta$, and proved that $\beta$ is a real  number. It remains to show $\beta > 1$. 

Let $\Re t > 0$, implying $\vert R_m(z(t)) \vert > 1$ and thus by A-stability $\Re z(t) > 0$. Choosing $\vert t\vert $ small enough implies that $\beta > 0$ and, due to the above expression for $\beta$, that in fact $\beta > 1$.
\end{proof}
We are now able to characterize $\{ \hat z_j(t,\theta) \}_{j=1}^m$ with the following proposition, that precludes two extremal cases.
\begin{proposition} \label{prop3}
The $m$ solutions of the equation
\begin{equation*}
	R_m(z) = e^{t + \mi \theta}
\end{equation*}
where $R_m(z)$ is stability function of the $m$-stage Runge-Kutta Gauss method, $t \in \mathbb{C}$ and $\theta \in [-\pi,\pi]$ with the exception $\theta=0$ if $m$ is even, and $\theta=\pm \pi$ if $m$ is odd, satisfy
\begin{equation*}
	\hat z_j(t,\theta) = \mi y_j(\theta) + \beta_j(\theta) t + \mathcal{O}(t^2)
\end{equation*}
when $\abs{t} \to 0$.
All the $\beta_j(\theta)$ and $y_j(\theta)$ are real numbers uniformly bounded in terms of $\theta$, when $\theta$ is distant from $\{-\pi, 0, \pi\}$. Moreover, it holds that $\beta_j(\theta) > 1$. 
\end{proposition}
\begin{proof}
We recall that
\begin{equation*}
	M_m(\theta,z) = P_m(z) - e^{\mi \theta} P_m(-z) = \sum_{j=0}^m p_j(1+(-1)^{j+1} e^{\mi\theta}) z^j.
\end{equation*}
The degree of $M_m(\theta,z)$ is exactly $m$ if and only if $\theta \ne 0$ for even $m$ and $\theta \ne \pm \pi$ for odd $m$. Since by hypothesis we are in that case, let denote by $\{\mi y_j(\theta)\}_{j=1}^m$, the $m$ zeros (in $z$) of $M_m(\theta,z)$. Note that $y_j(\theta) \in \mathbb{R}$ thanks to Proposition \ref{prop2}. Recalling standard estimates on the location of zeros of polynomials (see e.g. \cite[Theorem 27.2]{Marden1966}), we can bound
\begin{equation*}
	\max_{j=1,\ldots,m} \abs*{y_j(\theta)} \le 1+ \max_{j=1,\ldots,m} p_j \left\vert\frac{1+(-1)^{j+1} e^{\mi \theta}}{1+(-1)^{m+1}e^{\mi \theta}}\right\vert.
\end{equation*}
It is easy to see that, for each $0<d<\pi$, for $\theta \in [-\pi+d,-d] \cup [d,\pi-d]$, we obtain $\max_{j=1,\ldots,m} \abs{y_j(\theta)} \le M(m,d)$, where $M(m,d)>0$ does not depend on $\theta$. However, when $\theta$ is near $0$ if $m$ is even, or $\theta$ is near $\pm \pi$ if $m$ is odd, we can only bound, respectively
\begin{equation*}
	\max_{j=1,\ldots,m} \left\vert y_j(\theta) \right\vert = \mathcal{O}\left( \frac{1}{\theta} \right) \qquad \text{and} \qquad \max_{j=1,\ldots,m} \left\vert y_j(\theta) \right\vert = \mathcal{O}\left(\frac{1}{\pm \pi-\theta}\right).
\end{equation*}
Finally, thanks to Lemma \ref{lemma8}, the $m$ solutions of $R_m(z) = e^{t +i \theta}$ can be written as
\begin{equation*}
	\hat z_j(t,\theta) = \mi y_j(\theta) + \beta_j(\theta) t + \mathcal{O}(t^2), \quad  j = 1,\ldots,m
\end{equation*}
where $\beta_j(\theta)>1$ are the coefficients
\begin{equation*}
\beta_j(\theta) = \frac{1}{1-\frac{y_j(\theta)^{2m}}{\vert P_m(\mi y_j(\theta)) \vert^2}} .
\end{equation*}
\end{proof}
Now we consider the other two cases.
\begin{proposition} \label{prop4}
The solutions $\{\hat z_j(t,0) \}_{j=1}^m$ of $R_m(z) = e^{t}$, for $t \in \mathbb{C}$ when $\abs{t} \to 0$, can be characterized as
\begin{equation} \label{eq38}
	\hat z_1(t,0) = t + \mathcal{O}\bigl(t^{2m+1}\bigr), \quad \quad 
	\begin{cases}
		\hat z_{2\ell}(t,0) = \mi r_\ell + \delta_{\ell} t + \mathcal{O}\bigl(t^2\bigr),
		\\ \hat z_{2\ell+1}(t,0) = -\mi r_\ell + \delta_{\ell} t +  \mathcal{O}\bigl(t^2\bigr),
	\end{cases}
\end{equation}
for $\ell = 1, \ldots, \left\lceil \nicefrac{m}{2} \right\rceil-1$, where $r_\ell \in \mathbb{R}$, $\delta_\ell >1$. If $m$ is even, \eqref{eq38} gives the expression of $m-1$ solutions, whereas the last one satisfies
\begin{equation*}
	\hat z_{m}(t,0) = \frac{D_m}{t} + \mathcal{O}(1)
\end{equation*}
with $D_m>0$, when $\abs{t} \to 0$.
\end{proposition}
\begin{proof}
Recall that the solutions are the zeros of the polynomial
\begin{equation} \label{eq39}
	M_m(0,z) - tP_m(-z) + \mathcal{O}(t^2) P_m(z),
\end{equation}
where
\begin{equation*}
	M_m(0,z) = 2\sum_{\ell=1}^{\lceil\nicefrac{m}{2}\rceil} p_{2\ell-1} z^{2\ell-1}.
\end{equation*}
Again we look for $\{ \hat z_j(t,0) \}_{j=1}^m$ in the vicinity of the roots of the polynomial $M_m(0,z)$ of degree $2\lceil \nicefrac{m}{2} \rceil-1$. If $m$ is odd, then we find in this way all the solutions as in the previous case, if $m$ is even we are missing exactly one of them. One solution is always of the form $\hat{z}_1(t,0) = t + \mathcal{O}(t^{2m+1})$ since $R_m(t)$ is the $(m,m)$-Pad{\' e} approximant of $e^{t}$.
We write, for $\ell=1,\ldots, \lceil \nicefrac{m}{2} \rceil-1$, the solutions in pairs of the form, since in $M_m(0,z)$ appear only odd powers of $z$,
\begin{equation*} 
	\begin{cases}
		\hat z_{2\ell}(t,0) & = \mi r_\ell + \mathcal{O}(t),
		\\ \hat z_{2\ell+1}(t,0) & = -\mi r_\ell + \mathcal{O}(t),
	\end{cases}
\end{equation*}
where $\{\pm ir_\ell\}$ are the $2 \lceil \nicefrac{m}{2} \rceil-2$ zeros of $M_m(0,z)$ different from $0$. Moreover, we can characterize these $\lceil \nicefrac{m}{2} \rceil-1$ pairs as
\begin{equation*}
	\begin{cases}
		\hat z_{2\ell}(t,0) & = \mi r_\ell + \delta_\ell t + \mathcal{O}(t^2),
		\\ \hat z_{2\ell+1}(t,0) & = -\mi r_\ell + \delta_\ell t + \mathcal{O}(t^2),
	\end{cases}
\end{equation*}
for the same positive constants $\delta_\ell$. Indeed, thanks to Lemma \ref{lemma8}, it holds that
\begin{equation*}
	\delta_\ell = \frac{P_m(-ir_\ell)}{2 \sum_{j=1}^{\lceil \nicefrac{m}{2} \rceil} (2j-1) p_{2j-1} (\mi r_\ell)^{2j-2}} > 1.
\end{equation*}
Since $M_m(0,ir_\ell) = M_m(0,-ir_\ell)=0$, in $P_m(-ir_\ell)$ only even powers of $-ir_\ell$ appears, and we conclude that $\delta_\ell$ is the same shared constant both for $\hat z_{2\ell}(t,0)$ and $\hat z_{2\ell+1}(t,0)$.
We introduce a last type of solution that only appears for even $m$. We notice that the product of the zeros of the polynomial \eqref{eq39} is fixed, and this is equal to $p_0$, defined in \eqref{eq23}. From this, we deduce that
\begin{equation*}
	\hat z_m(t,0) = \frac{p_0}{t \prod_{\ell=1}^{\nicefrac{m}{2} }r_\ell^2} + \mathcal{O}(1) = \frac{D_m}{t} + \mathcal{O}(1)
\end{equation*}
with $D_m > 0$.
\end{proof}
A similar result holds for $\theta = \pm \pi$, the proof of which we omit for simplicity.
\begin{proposition} \label{prop5}
The solutions $\{\hat z_j(t,\pi) \}_{j=1}^m$ of $R_m(z) = -e^{t}$, for $t \in \mathbb{C}$ when $\abs{t} \to 0$ can be characterized as
\begin{equation} \label{eq47}
	\begin{cases}
		\hat z_{2\ell}(t,\pi) & = i \rho_\ell + \gamma_{\ell} t + \mathcal{O}(t^2),
		\\ \hat z_{2\ell+1}(t,\pi) & = -i \rho_\ell + \gamma_{\ell} t +  \mathcal{O}(t^2),
	\end{cases} \quad \ell = 0, \ldots, \left\lfloor \frac{m}{2} \right\rfloor,
\end{equation}
where $r_\ell \in \mathbb{R}$, $\gamma_\ell >1$. If $m$ is odd, \eqref{eq47} gives only the expression of $m-1$ solutions, and the last one satisfies
\begin{equation} \label{last_one}
	\hat z_{m}(t,\pi) = \frac{E_m}{t} + \mathcal{O}(1),
\end{equation}
with $E_m>0$, when $\abs{t} \to 0$.
\end{proposition}
We conclude this section with two technical lemmas regarding the ``conjugate'' pairs of solutions in \eqref{eq38} and \eqref{eq47}, and solution \eqref{last_one}, respectively.
\begin{lemma} \label{lemma2}
For $m \ge 3$ and for $\ell=1,\ldots, \lceil \nicefrac{m}{2} \rceil -1$, it holds 
\begin{equation*}
	R'_m(\hat z_{2\ell}(t,0)) -  R'_m(\hat z_{2\ell+1}(t,0)) = \mathcal{O}(t) \quad \text{and} \quad R'_m(\hat z_{2\ell}(t,0)) = \mathcal{O}(1)
\end{equation*} 
when $\abs{t} \to 0$. The same result holds as well for $ \{\hat z_{2\ell}(t,\pi), \hat z_{2\ell+1}(t,\pi) \}$ for $m \ge 2$ and $\ell = 0, \ldots, \lfloor \nicefrac{m}{2} \rfloor$.
\end{lemma}
\begin{proof}
We have the general formula
\begin{equation} \label{eq62}
	R'_m(z) = \frac{P_m'(z)P_m(-z) + P_m'(-z)P_m(z)}{P_m^2(-z)}.
\end{equation}
We only prove the result for $\theta=0$, first observing that $P_m(\mi r_\ell) = P_m(-\mi r_\ell)$ if $\pm \mi r_\ell$ are zeros of $M_m(0,z)$, for which easily we deduce $R_m'(\mi r_\ell) = R_m'(-\mi r_\ell)$. Recalling the characterizations of $\hat z_{2\ell}(t,0)$ and $\hat z_{2\ell+1}(t,0)$, the two claims easily follow.
\end{proof}
\begin{lemma} \label{lemma3}
The following holds
\begin{equation*}
	\begin{cases}
		R'_m(\hat z_m(t,0)) = \mathcal{O}(t^2) & \text{~for even~} m,
		\\ R'_m(\hat z_m(t,\pi)) = \mathcal{O}(t^2) & \text{~for odd~} m,
	\end{cases}
\end{equation*}
when $\abs{t} \to 0$.
\end{lemma}
\begin{proof}
The proof simply relies on the fact that $\hat z_m(t,0) = \nicefrac{D_m}{t} + \mathcal{O}(1)$ for even $m$, and $\hat z_m(t,\pi) = \nicefrac{E_m}{t} + \mathcal{O}(1)$ for odd $m$, and formula \eqref{eq62}.
\end{proof}
\section{Main results} \label{sec_Main}
\begin{lemma} \label{lemma4}
Let $K$ satisfy Assumption \ref{assumption1} and  consider an $m$-stage Gauss Runge-Kutta method satisfying Assumption \ref{assump2}. For every $\sigma_1 > \sigma_0$, there exist constants $\rho_1 > 0$ and $h_1 > 0$ such that for $0 < h \le h_1$ and all $s$ with $\Re s = \sigma_1$ and $\abs{sh} <\rho_1$,
\begin{align*} 
	b^T A^{-1} K\left( \frac{\Delta(e^{-sh})}{h} \right) e^{csh} \frac{e^{-sh}}{1-R_m(\infty) e^{-sh}}
	  &=  K(s) + s^{2m+\mu} \mathcal{O}(h^{2m}) \\ &
+
	\begin{cases}
		s^{2-\mu} \mathcal{O}(h^{2-2\mu}) & \text{~if~~} m = 2,  \\
		s^{m+1} \mathcal{O}(h^{m+1-\mu}) & \text{~if~~} m \ge 3 \text{~and~} m \text{~odd}, \\
		s^{m+1} \mathcal{O}(h^{m+1-\mu}) + s^{m-\mu} \mathcal{O}(h^{m-2\mu}) & \text{~if~~} m \ge 3 \text{~and~} m \text{~even}.
	\end{cases}
\end{align*}
The implied constants in the $\mathcal{O}$-notation are independent of $h$ and $s$.
\end{lemma}
\begin{proof}
Proceeding as in \cite{BanjaiLubichMelenk2011}, we write
\begin{align*}
	b^T A^{-1} & K\left( \frac{\Delta(e^{-sh})}{h} \right) e^{csh} \frac{e^{-sh}}{1-R_m(\infty) e^{-sh}} 
	\\ & = \frac{1}{2\pi i} \int_\Gamma K\left(\frac{z}{h}\right) b^T A^{-1} (z I - \Delta(e^{-sh}))^{-1} e^{csh} \frac{e^{-sh}}{1-R_m(\infty) e^{-sh}} \,\dd\Gamma_z
\end{align*}
where $\Gamma$ is a contour that encloses the spectrum of $\Delta(e^{-sh})$, which is composed by Proposition \ref{prop1} only of the solutions $\{\hat z_j(sh,0)\}_{j=1}^m$ of $R_m(z) = e^{sh}$. For simplicity, in the rest of the proof, we write $\hat z_j(sh)$ for $\hat z_j(sh,0)$. We recall that (see Lemma 2.6 of \cite{BanjaiLubichMelenk2011})
\begin{equation*}
	b^T A^{-1} \left(z I - \Delta(e^{-sh}) \right)^{-1} \frac{1}{1-R_m(\infty) e^{-sh}} = b^T A^{-1} (z I - A^{-1})^{-1}\frac{1}{1-R_m(z)e^{-sh}}.
\end{equation*}
So we have
\begin{align*}
	b^T A^{-1} K\left( \frac{\Delta(e^{-sh})}{h} \right) e^{csh} \frac{e^{-sh}}{1-R_m(\infty) e^{-sh}} = \frac{1}{2\pi i} \int_\Gamma K\left(\frac{z}{h}\right) b^T A^{-1} (z I - A^{-1})^{-1} e^{csh} \frac{e^{-sh}}{1-R_m(z)e^{-sh}} \, \dd \Gamma_z,
\end{align*}
where we can suppose that $\Gamma$ does not include the eigenvalues of $A^{-1}$. The integral is written as
\begin{align*}
	\frac{1}{2\pi i} & \int_\Gamma K\left(\frac{z}{h}\right) b^T A^{-1} (z I - A^{-1})^{-1} e^{csh} \frac{e^{-sh}}{1-R_m(z)e^{-sh}} \, \dd\Gamma_z
	\\ & = -\sum_{j=1}^m K\left( \frac{\hat z_j(sh)}{h} \right) b^T A^{-1} (\hat z_j(sh) I - A^{-1})^{-1} e^{csh} \frac{1}{R_q'(\hat z_j(sh))}.
\end{align*}
Using the characterization of Proposition \ref{prop4}, we calculate the residuals in all the $\{\hat z_j(sh)\}_{j=1}^m$.

\noindent \textbf{The first residual $\hat z_1(sh)$.} \\
Exactly as in \cite{BanjaiLubichMelenk2011}, since $\hat z_1(sh) = s h + \mathcal{O}((sh)^{2m+1})$ we deduce
\begin{equation*} 
	K\left( \frac{\hat z_1(sh)}{h} \right) b^T A^{-1} (\hat z_1(sh) I - A^{-1})^{-1} e^{csh} \frac{1}{R_m'(\hat z_1(sh))} = K(s) + s^{\mu} \mathcal{O}((sh)^{2m}).
\end{equation*}

\noindent \textbf{The pair of residuals $\hat z_{2\ell}(sh)$ and $\hat z_{2\ell+1}(sh)$ for $\ell = 1 , \ldots , \lceil \nicefrac{m}{2} \rceil-1$.} \\
We consider only the case $\hat z_{2\ell}(sh)$ for simplicity, the other is analogous. We recall that from Lemma 2.5 of \cite{BanjaiLubichMelenk2011}, for all $\alpha \in \mathbb{C}$ far from the inverse of the eigenvalues of $A$ it holds
\begin{equation} \label{eq74}
	(\alpha-z)b^T A^{-1} (\alpha I - A^{-1})^{-1} e^{cz} = e^z - R(\alpha) + \alpha b^T (I-\alpha A)^{-1}\mathcal{O}(z^{m+1})
\end{equation}
for $ \abs{z} \to 0$. We use the latter with $\alpha = \hat z_{2 \ell}(sh)$ and $z = sh$. The solutions $\hat z_{2 \ell}(sh)$ are uniformly bounded away from the eigenvalues of $A^{-1}$ by Proposition \ref{prop1}. Observing in our notation that $\alpha - z = \hat z_{2\ell}(sh) - sh = \mi r_\ell + (\delta_\ell-1)sh + \mathcal{O}((sh)^2)$, we conclude that
\begin{equation} \label{eq75}
	b^T A^{-1} (\hat z_{2\ell}(sh) I - A^{-1})^{-1} e^{csh} = b^T (I-\mi r_\ell A)^{-1}\mathcal{O}((sh)^{m+1}) = \mathcal{O}((sh)^{m+1}).
\end{equation}
Using Lemma \ref{lemma2}, and Assumption \ref{assumption1} it holds
\begin{equation*} 
	\left\| K\left( \frac{\hat z_{2 \ell}(sh)}{h} \right)  \frac{1}{R_q'(\hat z_{2\ell}(sh))} \right\|_{\mathcal{B}(X,Y)} \mathcal{O}((sh)^{m+1}) = s^{m+1} \mathcal{O}(h^{m+1-\mu}).
\end{equation*}
In the last step, we were allowed to use Assumption \ref{assumption1} since
\begin{equation*}
	\Re\left( \frac{\hat z_{2\ell}(sh)}{h} \right) = \delta_\ell \Re s + \mathcal{O}(h)
\end{equation*}
and by Proposition \ref{prop4}, $\delta_\ell >1$.

\noindent\textbf{The last residual $\hat z_m(sh)$ for $m$ even.} \\
 By virtue of \eqref{eq74}, we obtain 
\begin{align}
	& (\hat z_m(sh)-sh)b^T A^{-1} (\hat z_m(sh) I - A^{-1})^{-1} e^{csh} = \hat z_m(sh) b^T (I-\hat z_m(sh) A)^{-1}\mathcal{O}((sh)^{m+1}) = \mathcal{O}((sh)^{m+1}).  \label{eq78}
\end{align}
Since $\hat z_m(sh) = \nicefrac{D_m}{sh} + \mathcal{O}(1)$, we can use Assumption \ref{assumption1} for estimating $K$. In fact, the argument of $K$ is $\frac{\hat z_m(sh)}{h}$ whose real part is, up to lower order terms, $\nicefrac{D_m}{\abs{sh}^2} \Re s$ that clearly is in the half plane $\Re(\cdot) > \sigma_0$ for $\abs{sh}$ small enough. The latter observation, together with Lemma \ref{lemma3} and \eqref{eq78} implies 
\begin{align*} 
	\left\| K\left( \frac{\hat z_m(sh)}{h} \right) b^T A^{-1} (\hat z_m(sh) I - A^{-1})^{-1} e^{csh} \frac{1}{R'(\hat z_m(sh))} \right\|_{\mathcal{B}(X,Y)}  = s^{m-\mu} \mathcal{O}(h^{m-2\mu}).
\end{align*}
Combining the three type of residuals, we obtain the claim.
\end{proof}
\begin{remark} \label{remark2}
In some cases an even better result can be shown as we discuss in this remark, and we verify in an important example in the last section.
To show \eqref{eq74}, in \cite[Lemma 2.5]{BanjaiLubichMelenk2011}, it has been observed that using Runge-Kutta with time step $z$ for the problem $y'(s)=y(s)$ with $y(0)=1$ one obtains the vector
\begin{equation*}
	Y_0 = \mathbbm{1} + zAY_0,
\end{equation*}
and the stage order $Y_{0i} = e^{c_i z} + \mathcal{O}(z^{q+1})$, essentially, implies the result. We investigate the implicit constant vector in \eqref{eq74}. Writing
\begin{equation*}
	Y_0= (I-zA)^{-1} \mathbbm{1},
\end{equation*}
we deduce that
\begin{equation*}
	(I-zA)^{-1} \mathbbm{1} - e^{c z} = \mathcal{O}(z^{q+1}).
\end{equation*}
From this we can write explicitly
\begin{equation*}
	(I-zA)^{-1} \mathbbm{1} - e^{c z} = \sum_{j=q+1}^{\infty} z^j \left( A^j \mathbbm{1} - \frac{c^j}{j!}\right),
\end{equation*}
where the first non zero term is $C_{q} = A^{q+1} \mathbbm{1} - \nicefrac{c^{q+1}}{(q+1)!}$ with $c^{q+1} = (c_1^{q+1},\ldots,c_m^{q+1})$. Combining this observation with the proof of \cite[Lemma 2.5]{BanjaiLubichMelenk2011}, we deduce that $C_q$ is the implicit constant vector in \eqref{eq74}. We recall that for Gauss Runge-Kutta $q=m$.
For the pairs of residuals $\hat z_{2\ell}(sh)$ and $\hat z_{2\ell+1}(sh)$ for $\ell = 1 , \ldots , \lceil \nicefrac{m}{2} \rceil-1$, we expect some cancellation,  after using \eqref{eq75}, in the leading term of the sum
\begin{align}
	&\left\| K\left( \frac{\hat z_{2 \ell}(sh)}{h} \right) b^T (I-\mi r_\ell A)^{-1} C_m + K\left( \frac{\hat z_{2 \ell+1}(sh)}{h} \right)b^T (I+\mi r_\ell A)^{-1} C_m\right\|_{\mathcal{B}(X,Y)}  \mathcal{O}\bigl((sh)^{m+1}\bigr), \label{eq84}
\end{align}
when $K$ is such that
\begin{equation} \label{eq85}
	\left\| K\left( \frac{\mi r}{h} \right) + (-1)^m K\left( \frac{- i r}{h} \right)\right\|_{\mathcal{B}(X,Y)} = \mathcal{O}(1),
\end{equation}
for all fixed $r \in \mathbb{R}$ and $h \to 0$. In fact, it has been numerically verified that for $m \in \{3,\ldots,24\}$, for all $r_\ell$ as in Proposition \ref{prop4},
\begin{equation*}
	b^T\left[(I-\mi r_\ell A)^{-1} + (-1)^m (I+\mi r_\ell A)^{-1}\right]C_m = 0.
\end{equation*}
If $K$ is such that \eqref{eq85} holds true, then the error term related to the pair of residuals $\hat z_{2\ell}(sh)$ and $\hat z_{2\ell+1}(sh)$ can bounded by $s^{m+2}\mathcal{O}(h^{m+2-\mu})$ rather than the individual estimates $s^{m+1} \mathcal{O}(h^{m+1-\mu})$.
\end{remark}

We deduce similar results to Lemma \ref{lemma4} in the other two cases: when $\abs{sh}$ is near $\pi$ and when $\abs{sh}$ is bounded away both from $0$ and from $\pi$.
\begin{lemma}[$\abs{sh}$ near $\pi$] \label{lemma5}
Let $K$ satisfy Assumption \ref{assumption1} and consider an $m$-stage Gauss Runge-Kutta method satisfying Assumption \ref{assump2}. For every $\sigma_1 > \sigma_0$, there exist constants $\rho_2 > 0$ and $h_2 > 0$ such that for $0 < h \le h_2$ and all $s$ with
$\Re s = \sigma_1$ and $\rho_2 < \abs{sh} <\pi$,
\begin{align*} 
	& b^T A^{-1} K\left( \frac{\Delta(e^{-sh})}{h} \right) e^{csh} \frac{e^{-sh}}{1-R_m(\infty) e^{-sh}} =	
	\begin{cases} 
		\mathcal{O}(h^{-\mu}) + \mathcal{O}(h^{-1-2\mu}) & \text{~if~~} m \text{~is odd}, \\
		\mathcal{O}(h^{-\mu}) & \text{~if~~} m \text{~is even}.
	\end{cases}
\end{align*}
The implied constants in the $\mathcal{O}$-notation are independent of $h$ and $s$.
\end{lemma}
\begin{proof}
The proof is very similar to that one of Lemma \ref{lemma4}.
We need to estimate the residuals related to the solutions of equation $R_m(z) = e^{sh}$ obtained with the characterization in Proposition \ref{prop5}, where $t = \sigma_1 h$, with the supposition that $\abs{\omega h}$ is sufficiently near to $\pi$,  where we have written $s = \sigma_1 + i \theta$.  We have two type of residuals, the pairs and the last one that appears only for odd $m$.
\end{proof}
\begin{lemma}[$\abs{sh}$ bounded away from $0$ and $\pi$] \label{lemma6}
 Let $K$ satisfy Assumption \ref{assumption1} and let consider an $m$-stage Gauss Runge-Kutta method satisfying Assumption \ref{assump2}. For every $\sigma_1 > \sigma_0$, and for every $\rho_1, \rho_2 > 0$, there exists a constant $h_3 > 0$ such that for $0 < h \le h_3$ and all $s$ with
$\Re s = \sigma_1$ and $\rho_1 < \abs{sh} < \rho_2$,
\begin{equation*} 
	b^T A^{-1} K\left( \frac{\Delta(e^{-sh})}{h} \right) e^{csh} \frac{e^{-sh}}{1-R_m(\infty) e^{-sh}} =  \mathcal{O}(h^{-\mu}).
\end{equation*}
The implied constants in the $\mathcal{O}$-notation are independent of $h$ and $s$.
\end{lemma}
\begin{proof}
The proof of this Lemma is similar to that one of Lemma \ref{lemma4} and Lemma \ref{lemma5}. The only difference is that we need to estimate the residuals related to the solutions of equation $R_m(z) = e^{sh}$ obtained with the characterization in Proposition \ref{prop3}, taking $t = \sigma_1 h$ and $\theta = \omega h$, where we have written $s = \sigma_1 + i \theta$. We have only one type of residual, whose main term and first order term with respect to $t$ uniformly bounded in terms of $\theta$.
\end{proof}
We are ready to state and prove the main result.
\begin{theorem} \label{teorema1}
Let $K$ satisfy Assumption \ref{assumption1}. Consider an $m$-stage Runge-Kutta method based on the Gauss formulas satisfying Assumption \ref{assump2}. Let
\begin{equation*}
	r >
\begin{cases}
 \maxx\left\{4+\mu,2-\mu\right\} & \text{~if~~} m = 2, \\
   \maxx\left\{2m+\mu,m+1, m+2+\mu, 2m+2\mu+1\right\} & \text{~if~~} m \ge 3 \text{~and~} m \text{~odd}, \\ 
 \maxx\left\{2m+\mu, m+1, m-\mu\right\} & \text{~if~~} m \ge 3 \text{~and~} m \text{~even}.
	\end{cases}
\end{equation*}
Let $g \in C^r(X)$ with $g^{(r+1)} \in L^1_{\loc}(\mathbb{R},X)$, such that $g(0)=g^{(1)}(0)=\ldots=g^{(r-1)}(0)=0$. Then there exists $h_0$ such that for all $0 < h \le h_0$ and $t \in [0,T]$,
\begin{equation*}
	\left\| K(\partial_t^h)g(t) - K(\partial_t)g(t) \right\|_{Y} \apprle h^{p_{m,\mu}} \left( \| g^{(r)}(0) \|_X + \int_0^t \| g^{(r+1)}(\tau) \|_X \, \dd\tau \right),
\end{equation*}
where
\begin{equation*}
p_{m,\mu} =
	\begin{cases}
		\minn\{4,2-2\mu\} & \text{~if~~} m = 2, \\
		\minn\{2m,m+1-\mu\} & \text{~if~~} m \ge 3 \text{~and~} m \text{~odd}, \\
		\minn\{2m,m+1-\mu,m-2\mu\} & \text{~if~~} m \ge 3 \text{~and~} m \text{~even}.
	\end{cases}
\end{equation*}
The implicit constant is independent of $h$ and $g$, but does depend on $m$, $h_0$, $T$, and the constants in Assumption \ref{assumption1}.
\end{theorem}
\begin{proof} We proceed along the lines of \cite{Lubich1994}, Theorem 3.1. Applying the Laplace transform to the error
\begin{equation*}
	e_h(t) = K(\partial_t^h)g(t) - K(\partial_t)g(t)
\end{equation*}
yields
\begin{equation*}
	\mathcal{L}e_h(s) = \left\{b^T A^{-1} K\left( \frac{\Delta(e^{-sh})}{h} \right) e^{csh} \frac{e^{-sh}}{1-R_m(\infty) e^{-sh}} -  K(s)\right\} \mathcal{L}g(s).
\end{equation*}
Hence, if this expression is integrable along $\sigma_1 + i\mathbb{R}$, for some $\sigma_1 > \sigma_0$, with $\sigma_0$ as in Assumption \ref{assumption1}, we have by the inverse Laplace transform
\begin{align*}
e_h(t) = \frac{1}{2 \pi i} \int_{\sigma_1+i \mathbb{R}} e^{st} \left\{b^T A^{-1} K\left( \frac{\Delta(e^{-sh})}{h} \right) e^{csh} \frac{e^{-sh}}{1-R_m(\infty) e^{-sh}} - K(s)\right\} \mathcal{L}g(s) \, \dd s.
\end{align*}
Note that $\abs{e^{st} } = e^{\sigma_1 t} \le e^{\sigma_1 T}$ along the contour of integration. Let us first consider the special case $g(t) = \nicefrac{t^r}{r!}$, for which $\mathcal{L}g(s) = s^{-r-1}$, and study the integral
\begin{equation*}
	I = \int_{\mathbb{R}} \left\|{\left\{b^T A^{-1} K\left( \frac{\Delta(e^{-sh})}{h} \right) e^{csh} \frac{e^{-sh}}{1-R_m(\infty) e^{-sh}} - K(s)\right\} s^{-r-1}} \right\|_{\mathcal{B}(X,Y)} \dd\omega,
\end{equation*}
with $s=\sigma_1 + i \omega$. We split the integral into three parts:
\begin{align*}
	I & \le I_1 + I_2 + I_3 
	\\ & = \int_{\abs{sh}<\rho_1} \left\|{\left\{b^T A^{-1} K\left( \frac{\Delta(e^{-sh})}{h} \right) e^{csh} \frac{e^{-sh}}{1-R_m(\infty) e^{-sh}} -  K(s)\right\} s^{-r-1}} \right\|_{\mathcal{B}(X,Y)} \dd\omega 
	\\ & +\int_{\abs{sh}\ge \rho_1} \left\|{b^T A^{-1} K\left( \frac{\Delta(e^{-sh})}{h} \right) e^{csh} \frac{e^{-sh}}{1-R_m(\infty) e^{-sh}} s^{-r-1}} \right\|_{\mathcal{B}(X,Y)} \dd\omega
	\\ & +\int_{\abs{sh}\ge \rho_1} \left\| K(s) s^{-r-1} \right\|_{\mathcal{B}(X,Y)} \dd\omega,
\end{align*}
with the constant $\rho_1$ from Lemma \ref{lemma4}. By virtue of Assumption \ref{assumption1}, we obtain, for $r > \mu$
\begin{equation*}
	I_3 \apprle \int_{\abs{sh} \ge \rho_1} \left\vert s \right\vert^{\mu-r-1} \dd\omega = \mathcal{O}(h^{r-\mu}).
\end{equation*}
By periodicity of the exponential function, one may assume that $\vert \Im (s h) \vert = \vert \omega h \vert \le \pi$ in $I_2$, then the integral can be written (as in \cite{Lubich1994}) 
\begin{align*}
	I_2 & \apprle \int_{\abs{wh} \le \pi} \left\| b^T A^{-1} K\left( \frac{\Delta(e^{-sh})}{h} \right) e^{csh} \frac{e^{-sh}}{1-R_m(\infty) e^{-sh}} \right\|_{\mathcal{B}(X,Y)} \sum_{n \ne 0} \left\vert s + \frac{2 \pi i n}{h} \right\vert^{-r-1} \dd\omega
	\\ & \apprle h^{r+1} \int_{\abs{wh} \le \pi} \left\| b^T A^{-1} K\left( \frac{\Delta(e^{-sh})}{h} \right) e^{csh} \frac{e^{-sh}}{1-R_m(\infty) e^{-sh}} \right\|_{\mathcal{B}(X,Y)} \dd\omega.
\end{align*}

Now we distinguish three cases: when $\abs{\omega h}$ is small enough and $m$ is even, when $\abs{\omega h}$ is near enough to $\pi$ for $m$ odd, and the other cases. We take $\rho_1$ and $\rho_2$ from Lemma \ref{lemma4} and Lemma \ref{lemma5} respectively, and we split
\begin{align*}
	I_2  & \apprle h^{r+1} \int_{\abs{\omega h} < \rho_1} \left\| b^T A^{-1} K\left( \frac{\Delta(e^{-sh})}{h} \right) e^{csh} \frac{e^{-sh}}{1-R_m(\infty) e^{-sh}} \right\|_{\mathcal{B}(X,Y)} \dd \omega 
	\\ & +  h^{r+1} \int_{\rho_1 < \abs{\omega h} < \rho_2} \left\| b^T A^{-1} K\left( \frac{\Delta(e^{-sh})}{h} \right) e^{csh} \frac{e^{-sh}}{1-R_m(\infty) e^{-sh}} \right\|_{\mathcal{B}(X,Y)} \dd \omega 
	\\ & +  h^{r+1} \int_{\rho_2 < \abs{\omega h} < \pi} \left\| b^T A^{-1} K\left( \frac{\Delta(e^{-sh})}{h} \right) e^{csh} \frac{e^{-sh}}{1-R_m(\infty) e^{-sh}} \right\|_{\mathcal{B}(X,Y)} \dd \omega
	\\ & = Y_1 + Y_2 + Y_3.
\end{align*}
For the first range, that we take into account only when $m$ is even, we can bound using Lemma \ref{lemma4}
\begin{align*}
	Y_1 & \apprle h^{r+1} \int_{\abs{\omega h} < \rho_1} \left[ \abs{s}^{\mu} + h^{2m} \abs{s}^{\mu+2m} + h^{m+1-\mu} \abs{s}^{m+1} + h^{m-2\mu} \abs{s}^{m-\mu} \right] \dd\omega
	 \\ & = 
	 \begin{cases}
		\mathcal{O}(h^{r+1}) & \mu < -1, \\
		\mathcal{O}(h^{r+1}  \log{h})  & \mu=-1, \\
		\mathcal{O}(h^{r+1-\mu}) & -1 < \mu \le m+1, \\
		\mathcal{O}(h^{r+1+m-2\mu}) & m+1 < \mu.
	\end{cases}
\end{align*}
For the second range, for all $m$, we obtain using Lemma \ref{lemma6}
\begin{equation*}
	Y_2 \apprle h^{r+1-\mu} \int_{\rho_1 < \abs{\omega h} < \rho_2}  \dd\omega = \mathcal{O}(h^{r-\mu}).
\end{equation*}
Finally, for the third range, for $m$ odd, we bound, using Lemma \ref{lemma5},
\begin{equation*}
	Y_3 \apprle h^{r+1} \int_{\rho_2 < \abs{\omega h} < \pi} \left[ h^{-\mu} + h^{-2\mu-1} \right] \dd\omega  = \mathcal{O}(h^{r-\mu}) + \mathcal{O}(h^{r-2\mu-1})
\end{equation*}
It remains to estimate $I_1$. Using again Lemma \ref{lemma4} we have
\begin{align*}
	I_1 & \apprle  \int_{\abs{s h} < \rho_1} h^{2m} \abs{s}^{\mu+2m-r-1}\dd\omega +
	\begin{cases}
	\int_{\abs{s h} < \rho_1} h^{2-2\mu}\abs{s}^{m-\mu-r-1} \dd\omega & \text{~if~~} m = 2, \\
	 	\int_{\abs{s h} < \rho_1} h^{m+1-\mu} \abs{s}^{m-r}  \dd\omega & \text{~if~~} m \ge 3 \text{~and~} m \text{~odd}, \\
	\int_{\abs{s h} < \rho_1} h^{m+1-\mu} \abs{s}^{m-r} + h^{m-2\mu} \abs{s}^{m-\mu-r-1}  \dd\omega & \text{~if~~} m \ge 3 \text{~and~} m \text{~even},
	\end{cases} 
\\ &	=
	 \mathcal{O}(h^{2m}) +
	 \begin{cases}
	 	\mathcal{O}(h^{2-2\mu}) & \text{~if~~} m = 2, \\
		 \mathcal{O}(h^{m+1-\mu})  & \text{~if~~} m \ge 3 \text{~and~} m \text{~odd}, \\
		 \mathcal{O}(h^{m+1-\mu}) + \mathcal{O}(h^{m-2\mu}) & \text{~if~~} m \ge 3 \text{~and~} m \text{~even},
	 \end{cases}
\end{align*}
where we have used the hypothesis that
\begin{equation} \label{rsup}
r > 
	\begin{cases}
		\max\{4+\mu, 2-\mu\} & m = 2, \\
		\max\{\mu+2m, m+1\} & m \ge 3 \text{~and~} m \text{~odd}, \\
		\max\{\mu+2m, m+1, m-\mu\} & m \ge 3 \text{~and~} m \text{~even}. \\
	\end{cases}
\end{equation} 
For even $m$ is easy to see that the error terms from $I_2$ and $I_3$ are always smaller than those from $I_1$, for $r$ satisfying \eqref{rsup}. However, in the case $m$ odd, we have also the terms from $Y_3$, and we need to require also $r > \max\{m+2+\mu,2m+2\mu+1\}$. For general $g$ we can proceed as in \cite{BanjaiLubichMelenk2011}.
\end{proof}
Concluding, we have in general a better order of convergence for odd $m$, on the other hand a lower regularity requirement on $g$ for even $m$. 
\subsection{Numerical example}
Let us consider the scalar case, as in \cite{BanjaiLubich2011}, with the kernel 
\begin{equation} \label{eq111}
	K_{\mu}(s) = \frac{s^{\mu}}{1-e^{-s}}.
\end{equation}
We approximate the convolution $K_{\mu}(\partial_t) g$ by the convolution quadrature based on 2-stage and 3-stage Gauss methods, with
\begin{equation} \label{eq112}
	g(t) = e^{-0.4t} \sin^6(t),
 \end{equation}
 and final computational time $T = 3$. We use this example to illustrate the sharpness of Theorem \ref{teorema1}. When $m=2$ we expect the convergence order $h^{\min\{4,2-2\mu\}}$, when $m=3$, order $h^{\min\{6,4-\mu\}}$. The relative discrete $l^2$ norm is calculated with respect to a reference solution obtained with $N_{t_{ref}}=2048$ time steps. The results in Tables \ref{table1} and \ref{table2}  confirm that the convergence rates we have proved are also optimal, with the exception of $m=3$ and $\mu=1$. However, in the latter case, since for these parameters the kernel satisfies \eqref{eq85}, we notice the superconvergence due to the cancellation phenomena already predicted in Remark \ref{remark2}. In the next section we illustrate an important application where we observe the same superconvergence.

\begin{table}[h]
  \centering
{
\begin{tabular}{|c|c|c|c|c|c|c|c|c|}\hline
 $N_t$ &  $\mu=-1$ & EOC & $\mu=0$ & EOC & $\mu=1$ & EOC \\\hline
$16$ & $1.2e$-$04$ &  & $3.6r$-$03$ &  & $4.2e$-$01$ &    \\
& &3.9& & 2.1 & & 0.0    \\
$32$ & $8.2e$-$06$ &  & $8.6e$-$04$ &  & $4.4e$-$01$ &  \\
& &4.0& & 2.0 & & 0.0   \\
$64$ & $5.2e$-$06$ &  & $2.1e$-$04$ &  & $4.4e$-$01$ &   \\
& &4.0& & 2.0 & & 0.0    \\
$128$ & $3.3e$-$07$ &  & $5.3e$-$05$ &  & $4.4e$-$01$ &   \\
& &3.9& & 2.0 & & 0.0    \\
$256$ & $2.2e$-$08$ &  & $1.3e$-$05$ &  & $4.4e$-$01$ &   \\
\hline
\end{tabular}}
\caption{Runge-Kutta based on Gauss method with stage order $m=2$ for kernel \eqref{eq111} and datum \eqref{eq112}, by varying $\mu \in \left\{-1, 0,1\right\}$.}
\label{table1}
\end{table}
\begin{table}[h]
  \centering
{
\begin{tabular}{|c|c|c|c|c|c|c|c|c|}\hline
 $N_t$ &  $\mu=0$ & EOC & $\mu=\nicefrac{1}{2}$ & EOC & $\mu=1$ & EOC \\\hline
$16$ & $8.8e$-$05$ &  & $8.0e$-$04$ &  & $1.5e$-$02$ &    \\
& &4.2& & 4.2 & & 4.2    \\
$32$ & $4.8e$-$06$ &  & $4.5e$-$05$ &  & $8.1e$-$03$ &  \\
& &4.0& & 3.7 & & 4.0    \\
$64$ & $3.0e$-$07$ &  & $3.5e$-$06$ &  & $4.9e$-$04$ &   \\
& &4.0& & 3.6 & & 3.9    \\
$128$ & $1.9e$-$08$ &  & $3.0e$-$07$ &  & $3.2e$-$05$ &   \\
& &3.9& & 3.4 & & 3.8    \\
$256$ & $1.2e$-$09$ &  & $2.7e$-$08$ &  & $2.4e$-$06$ &   \\
\hline
\end{tabular}
}
  \caption{Runge-Kutta based on Gauss method with stage order $m=3$ for kernel \eqref{eq111} and datum \eqref{eq112}, by varying $\mu \in \left\{0, \nicefrac{1}{2},1\right\}$.}
\label{table2}
\end{table}
\section{Application to acoustic wave propagation} \label{sec_wave}
Let $\Omega \subset \mathbb{R}^d,$ $d=2$ or $d=3$, a bounded Lipschitz domain with boundary $\Gamma = \partial \Omega$. 
 For $\Re s>0$, the single- and double-layer potentials for the Helmholtz equation 
 \begin{equation} \label{eq113}
 -\Delta U(x) + s^2 U(x) = 0, \quad x \in \mathbb{R}^d \setminus \Gamma
 \end{equation}
 are defined by 
\begin{align*}
	S(s) \phi(x) & = \int_\Gamma K(x-y,s) \phi(y) \, \dd\Gamma_y, \quad x \in \mathbb{R}^d \setminus \Gamma
	\\ D(s) \phi(x)  & = \int_\Gamma \partial_{n_y} K(x-y,s) \phi(y) \, \dd\Gamma_y, \quad x \in \mathbb{R}^d \setminus \Gamma
\end{align*}
where $K$ is the fundamental solution of the operator $-\Delta \cdot + s^2 \cdot$. The latter is given by
\begin{equation*}
	K(x,s) = 
	\begin{cases}
			\vspace{0.2cm}
		\frac{1}{4} K_0(s \| x \|) & d=2, \\
		\frac{e^{-s\|x\|}}{4\pi \|x\|} & d=3 \\
	\end{cases}
\end{equation*}
where $K_0$ is the modified Bessel function of second kind of order 0. The single-layer potential is continuous across $\Gamma$, we denote its boundary trace by
\begin{equation*}
	V(s) \phi(x) = \int_\Gamma K(x-y,s) \phi(y) \, \dd\Gamma_y, \quad x \in \Gamma.
\end{equation*}
We recall from \cite{BambergerHaDuong1986} that $V(s)$ is invertible and
\begin{equation*}
	\| V^{-1} (s) \|_{\mathcal{B}(H^{\frac{1}{2}}(\Gamma),H^{-\frac{1}{2}}(\Gamma))} \le M \frac{\abs{s}^2}{\Re s} \quad \text{for~} \Re s \ge \sigma_0 > 0,
\end{equation*}
for a constant $M>0$.
We define the corresponding boundary operator for the double-layer potential as the average of the two traces on $\Gamma$
\begin{equation*}
	K(s) \phi(x) = \frac{1}{2} (\gamma^+ + \gamma^-)D(s) \phi(x), \quad x \in \Gamma.
\end{equation*}
The exterior Dirichlet-to-Neumann operator $DtN^+(s) : H^{\nicefrac{1}{2}}(\Gamma) \to H^{-\nicefrac{1}{2}}(\Gamma)$ is defined as $DtN^+(s)g = \partial_n^+U$ where $U$ solves \eqref{eq113} with the Dirichlet condition $\gamma^+ U = g$ on $\Gamma$. In \cite{BambergerHaDuong1986, LalienaSayas2009} it was proved that 
\begin{equation*}
DtN^+(s) = V^{-1}(s)\left(-\frac{1}{2} + K(s)\right),
\end{equation*}
and that it holds
\begin{equation*}
	\| DtN^+ (s) \|_{\mathcal{B}(H^{\nicefrac{1}{2}}(\Gamma),H^{-\nicefrac{1}{2}}(\Gamma))} \le M \frac{\abs{s}^2}{\Re s} \quad \text{for~ } \, \Re\, s \ge \sigma_0 > 0.
\end{equation*}
However, as proved in \cite{Banjai2014} for the case $\Omega = \mathbb{S}^2$ and argued for general convex domains, the following stronger estimate holds true in some situations
\begin{equation} \label{eq122}
	\| DtN^+ (s) + s \|_{\mathcal{B}(H^{\nicefrac{1}{2}}(\Gamma),H^{-\nicefrac{1}{2}}(\Gamma))} \le C \quad  \text{for~} \Re s \ge \sigma_0 > 0,
\end{equation}
for $C>0$.
In \cite{MelenkRieder2020}, the same result has been proved for $s$ in a sector of the form $\abs{\arg s} \le \nicefrac{\pi}{2} - c$ for $c \in (0, \nicefrac{\pi}{2})$, for all Lipschitz domains $\Omega$. In the same paper, the property
\begin{equation} \label{eq123}
\| V^{-1}(s) - 2s \|_{\mathcal{B}(H^{\nicefrac{1}{2}}(\Gamma),H^{-\nicefrac{1}{2}}(\Gamma))} \le C
\end{equation}
has been proved for $s$ in a sector as before. It is still an open question whether \eqref{eq122} or \eqref{eq123} hold for all Lipschitz domains in a half-plane of the type $\Re s \ge \sigma > 0$. 

Assuming that we have a Lipschitz domain for which \eqref{eq122} or \eqref{eq123} holds in the whole right-half complex plane, then we are able to prove a stronger convergence theorem related to the $DtN^+$ or $V^{-1}$ operators, when the number of stages of the Gauss Runge-Kutta method is odd.
\begin{corollary} \label{corollary1}
Consider an $m$-stage Gauss Runge-Kutta method satisfying Assumption \ref{assump2} with odd $m \in \mathbb{N}$. Moreover, suppose that the Runge-Kutta method is such that
\begin{equation*} 
	b^T\left[(I-\mi r_\ell A)^{-1} + (-1)^m (I+\mi r_\ell A)^{-1}\right]C_m = 0
\end{equation*}
for all $\ell = 1, \ldots, \nicefrac{(m+1)}{2} $, where $r_\ell$ are defined in Proposition \ref{prop4} and $C_m$ in Remark \ref{remark2}.
Let $r >  2m+3$, and let $\Omega$ be a Lipschitz domain with boundary $\Gamma = \partial \Omega$, such that it holds \eqref{eq122} or \eqref{eq123} for all $\Re s \ge \sigma_0 >0$. Let $g \in C^r(H^{\nicefrac{1}{2}}(\Gamma))$ with $g^{(r+1)} \in L^1_{\loc}(\mathbb{R},H^{\nicefrac{1}{2}}(\Gamma))$, such that $g(0)=g^{(1)}(0)=\ldots=g^{(r-1)}(0)=0$. Then, there exists $h_0$ such that for all $0 < h \le h_0$ and $t \in [0,T]$,
\begin{align*}
	\left\| K(\partial^h_t)g(t) - K(\partial_t)g(t) \right\|_{H^{-\nicefrac{1}{2}}(\Gamma)}  \apprle h^{m+1} \left( \| g^{(r)}(0) \|_{H^{\nicefrac{1}{2}}(\Gamma)} + \int_0^t \| g^{(r+1)}(\tau) \|_{H^{\nicefrac{1}{2}}(\Gamma)} \dd\tau \right),
\end{align*}
where $K = DtN^+$, if \eqref{eq122} holds, or $K = V^{-1}$, if \eqref{eq123} holds.
The implicit constant is independent of $h$ and $g$, but does depend on $m$, $h_0$, $T$, $\sigma_0$ and $C$.
\end{corollary}
\begin{proof}
The prove simply combines Theorem \ref{teorema1}, Remark \ref{remark2} and hypothesis \eqref{eq122}/\eqref{eq123}. In fact, we are able to gain one extra order of convergence thanks to the cancellation in the error \eqref{eq84} related to the pairs of residuals $\hat z_{2\ell}(sh)$ and $\hat z_{2\ell+1}(sh)$.
\end{proof}
\subsection{Numerical results}
We consider the scattering of an incident wave by the unitary circle centered in the origin and a L-shaped domain, which corners are given by $(1,0.1), (0.1, 0.1), (0.1, 1), (-1, 1), (-1, -1)$ and $(1,-1)$. For the space discretization we employ a piecewise-constant Galerkin boundary element method, with a fixed fine spatial discretization of $128$ points on $\Gamma$. 
In order to make sure that the space discretization does not significantly affect the results, we have computed all the results with a finer space discretization of $256$ points of discretization, this computation gave essentially the same result.
Since the analytic solutions are not known we have estimated the errors by the following
\begin{equation} \label{eq140}
	\text{error}_{N_t} = \left( h \sum_{j=0}^{N_t} \left\| \varphi_{N_t}(t_j) - \varphi_{N_{t_{ref}}}(t_j) \right\|^2_{H^{-\nicefrac{1}{2}}(\Gamma)} \right)^{\nicefrac{1}{2}}
\end{equation}
calculated using the norm equivalence
\begin{equation*}
	\| \varphi \|^2_{H^{-\nicefrac{1}{2}}(\Gamma)} \approx \langle V(1) \varphi , \varphi \rangle_{\Gamma}
\end{equation*}
where $\langle \cdot, \cdot \rangle_{\Gamma}$ is the duality paring between $H^{\nicefrac{1}{2}}$ and $H^{-\nicefrac{1}{2}}$. In \eqref{eq140}, $\varphi_{N_t}$ is the discrete solution obtained by convolution quadrature with time step $h = \nicefrac{T}{N_t}$. The reference solution $\varphi_{N_{t_{ref}}}$ is calculated with a Runge-Kutta 5-stage Radau IIA convolution quadrature, and $N_{t_{ref}} = 210$.

For the first experiment we discretize, only for the circle domain, the boundary integral equation $\varphi = V^{-1}(\partial_t)g$, with the benchmark right hand side
\begin{equation} \label{eq142}
	g(x_1,x_2,t) = \left(1+\sin(x_2)^2\right)t^{15}.
\end{equation}
The final time of investigation is $T=1$. We report the results in Table \ref{table3}, which are coherent with Corollary \ref{corollary1}.
\begin{table}[h]
  \centering
{
\begin{tabular}{|c|c|c|c|c|c|c|c|c|}\hline
 $N_t$ &  $m=2$ & EOC & $ m=3$ & EOC & $m=4$ & EOC & $m=5$ & EOC \\\hline
$6$ & $1.6e$+$01$ &  & $2.6e$-$00$ &  & $9.9e$-$01$ &  & $4.9e$-$02$ &   \\
& &1.0& & 3.6 & & 2.7 & & 5.8   \\
$7$ & $1.4e$+$01$ &  & $1.5e$-$00$ &  & $6.5e$-$01$ &  & $2.0e$-$02$ &   \\
& &0.8& & 3.9 & & 2.6 & & 6.0   \\
$10$ & $1.1e$+$01$ &  & $3.8e$-$01$ &  & $2.6e$-$01$ &  & $2.4e$-$03$ &   \\
& &0.6& & 4.1 & & 2.5 & & 6.0   \\
$14$ & $8.7e$-$00$ &  & $9.6e$-$02$ &  & $1.1e$-$02$ &  & $3.2e$-$04$ &  \\
& &0.5& & 4.1 & & 2.4 & & 5.3   \\
$15$ & $8.4e$-$00$ &  & $7.2e$-$02$ &  & $9.5e$-$02$ &  & $2.2e$-$04$ &   \\
& &0.4& & 4.1 & & 2.3 & & 4.6   \\
$21$ & $7.3e$-$00$ &  & $1.8e$-$02$ &  & $4.3e$-$02$ &  & $4.7e$-$05$ &   \\
\hline
\end{tabular}
}
  \caption{Gauss Runge-Kutta method with stage order $m$ for the discretization of $V^{-1}$ with datum \eqref{eq142}.}
  \label{table3}
\end{table}

For our second experiment, both for the circle and for the L-shaped domain described before, we discretize $\varphi = DtN^+(\partial_t)g$. We consider the right hand side
\begin{equation} \label{second_datum}
	g(x,t) = f\left(\frac{1}{\rho}(t-x \cdot \alpha +A)\right), \quad f(t) = e^{-t^2}, 
\end{equation}
where $\rho = \frac{3}{8}, \alpha = \left[-\nicefrac{1}{\sqrt{2}}, -\nicefrac{1}{\sqrt{2}}\right],$ and $A = -4$.
For this test, we take $T=3$. The results of these numerical experiments, as documented in Table \ref{table4} and \ref{table5}, suggest a convergence order $\mathcal{O}(h^3)$ when using the 3-stage Radau IIA method and a convergence order $\mathcal{O}(h^4)$ using the 3-stage Gauss method. These results are coherent with Theorem 3.2 of \cite{BanjaiLubichMelenk2011} and Corollary \ref{corollary1}, even if at the moment the estimate \eqref{eq122} is not known for the L-shaped domain.
\begin{table}[h]
  \centering
{
\begin{tabular}{|c|c|c|c|c|c|c|c|c|c|c|c|}\hline
 $N_t$ & 15 & & 21 & & 35 & & 42 & & 70 \\\hline
Gauss      & $2.8e$-$02$ &             & $4.6e$-$03$ &              & $4.6e$-$04$ &             & $2.2e$-$04$ &              & $2.6e$-$05$ \\
EOC         &                    & $5.4$ &                     & $4.5$ &                    & $4.2$ &                    & $4.1$ & \\
\hline 
Radau IIA & $1.2e$-$02$ &            & $4.5e$-$03$  &             & $1.0e$-$03$ &             & $5.8e$-$04$ &             & $1.3e$-$04$ \\
EOC         &                    & $2.9$ &                     & $3.0$ &                    & $3.0$ &                    & $3.0$ & \\
\hline 
\end{tabular}
\vspace{0.01cm}
}
\caption{Convergence of the 3-stage Radau IIA and Gauss convolution quadrature methods of the $DtN^+$ operator for the circle, for datum \eqref{second_datum}.}
\label{table4}
\end{table}
\begin{table}[h]
  \centering
{
\begin{tabular}{|c|c|c|c|c|c|c|c|c|c|c|c|}\hline
 $N_t$ & 15 & & 21 & & 35 & & 42 & & 70 \\\hline
Gauss      & $8.2e$-$03$ &             & $1.8e$-$03$ &              & $1.9e$-$04$ &             & $8.6e$-$05$ &              & $1.1e$-$06$ \\
EOC         &                    & $4.6$ &                     & $4.4$ &                    & $4.2$ &                    & $4.1$ & \\
\hline 
Radau IIA & $7.0e$-$03$ &            & $2.5e$-$03$  &             & $5.5e$-$04$ &             & $3.1e$-$04$ &             & $6.8e$-$05$ \\
EOC         &                    & $3.0$ &                     & $3.0$ &                    & $3.0$ &                    & $3.0$ & \\
\hline 
\end{tabular}
\vspace{0.01cm}
}
\caption{Convergence of the 3-stage Radau IIA and Gauss convolution quadrature methods, of the $DtN^+$ operator for the L-shaped domain, for datum \eqref{second_datum}.}
\label{table5}
\end{table}

\section{Conclusion}
In this paper, we have proposed and analyzed a numerical approach for computing integrals of the form \eqref{first_eq} by means of a convolution quadrature based on Gauss Runge-Kutta methods. The analysis performed is based on the localization of zeros of complex polynomials associated to Pad{\' e} approximants of the exponential.
When the number of stages $m$ is even, the convergence rates are worse than those obtained for stiffly accurate Runge-Kutta methods (such as Radau IIA and Lobatto IIIC).On the other hand, when $m$ is odd, Theorem \ref{teorema1} guarantees the same convergence as in the case of CQ based on stiffly accurate Runge-Kutta methods (see \cite[Theorem 4.1]{BanjaiLubich2011}) for hyperbolic kernels satisfying Assumption \ref{assumption1}. In the case of kernels satisfying the refined bound in Remark \ref{remark1}, a better convergence rate holds for stiffly accurate Runge-Kutta methods (see \cite{BanjaiLubichMelenk2011}). In contrast, it does not seem to be possibile to use this more refined bound in the analysis for Gauss Runge-Kutta methods irrespective of the parity of $m$. However, in the odd case, we establish an improvement, even with respect to Radau IIA, for kernels satisfying condition \eqref{eq85}. In particular, we show that in certain situations, including the exterior Dirichlet-to-Neumann map, an extra order of convergence with respect to those based on Radau IIA can occur. This property, for the $DtN^+$ map, heavily relies on inequality \eqref{eq122} that, at the moment, has been proved only when the scattering obstacle is a sphere or the whole right half plane.

\section*{Declaration}
This research was performed while M.F. was visiting L.B. at Heriot-Watt University. The second author was partially supported by MIUR grant Dipartimenti di Eccellenza 2018-2022, CUP E11G18000350001.


\begin{thebibliography}{30}

%
\bibitem{BambergerHaDuong1986}
Bamberger, A., Ha Duong T., Formulation variationnelle espace-temps pour le calcul par potentiel retard\'{e} de la diffraction d'une onde acoustique. {I}, Mathematical Methods in the Applied Sciences, \textbf{8} (3) (1986), 405-435.
%
\bibitem{Banjai2014}
Banjai L., Time-domain {D}irichlet-to-{N}eumann map and its discretization, IMA Journal of Numerical Analysis, \textbf{34} (3) (2014), 1136-1155.
%
\bibitem{BanjaiLubich2011}
Banjai L., Christian L., An error analysis of {R}unge-{K}utta convolution quadrature, BIT. Numerical Mathematics, \textbf{51} (3) (2011), 483-496.
%
\bibitem{BanjaiLubichMelenk2011}
Banjai, L., Lubich, C., Melenk, J. M., Runge-{K}utta convolution quadrature for operators arising in wave propagation, Numerische Mathematik, \textbf{119} (1) (2011), 1-20.
%
\bibitem{BanjaiSayas2022}
Banjai, L., Sayas, F. J., Integral Equation Methods for Evolutionary PDE: A Convolution Quadrature Approach, Springer Series in Computational Mathematics, (59) Springer, Cham (2022).
%
\bibitem{Butcher1964}
Butcher, J. C., Implicit {R}unge-{K}utta processes, Mathematics of Computation, \textbf{18}, (1964), 50-64.
%
\bibitem{CalvoCuestaPalencia2007}
Calvo, M. P., Cuesta, E., Palencia, C., Runge-{K}utta convolution quadrature methods for well-posed equations with memory, Numerische Mathematik, \textbf{107} (4) (2007), 589-614.
%
\bibitem{CostabelSayas2017}
Costabel, M., Sayas, F. Time-dependent problems with the boundary integral equation method. In: Stein, E., Borst, R., Hughes, T. J. (eds.) Encyclopedia of Computational Mechanics, 2nd edn., Wiley, Hoboken (2017).
%
\bibitem{Ehle1973}
Ehle, B. L., {{$A$}-stable methods and {P}ad\'{e} approximations to the exponential}, SIAM Journal on Mathematical Analysis, \textbf{4} (1973), 671-680.
%
\bibitem{HairerWanner1991}
Hairer, E., Wanner, G., Solving ordinary differential equations. {II}, Springer Series in Computational Mathematics, 14, Springer-Verlag, Berlin, (1991).
%
\bibitem{LalienaSayas2009}
Laliena, A. R., Sayas, F.-J., {Theoretical aspects of the application of convolution quadrature to scattering of acoustic waves}, Numerische Mathematik, \textbf{112} (4) (2009), 637-678.
%
\bibitem{Lubich1988a}
Lubich, C., Convolution quadrature and discretized operational calculus. I, Numerische Mathematik, \textbf{52} (2) (1988), 129-145.
%
\bibitem{Lubich1988b}
Lubich, C., Convolution quadrature and discretized operational calculus. II, Numerische Mathematik, \textbf{52} (4) (1988), 413-425.
%
\bibitem{Lubich1994}
Lubich, C., On the multistep time discretization of linear initial-boundary value problems and their boundary integral equations, Numerische Mathematik, \textbf{67} (3) (1994), 365-389.
%
\bibitem{LubichOstermann1993}
Lubich, C., Ostermann, A., Runge-{K}utta methods for parabolic equations and convolution quadrature, Mathematics of Computation, \textbf{60} (1993), 105-131.
%
\bibitem{Marcus1990}
Marcus, M., Determinants of Sums, The College Mathematics Journal \textbf{21} (2) (1990), 130-135.
%
\bibitem{Marden1966}
Marden, M., Geometry of polynomials, Mathematical Surveys, No. 3. American Mathematical Society, Providence, R.I. (1966).
%
\bibitem{MelenkRieder2020}
Melenk, J. M., Rieder A., On superconvergence of {R}unge-{K}utta convolution quadrature for the wave equation, Numerische Mathematik \textbf{147} (1) (2021), 157-188.
%
\bibitem{SaffVarga1975}
Saff, E. B. and Varga, R. S., On the zeros and poles of {P}ad\'{e} approximants to {$e^{z}$}, Numerische Mathematik, \textbf{25} (1) (1975), 1-14.

\end{thebibliography}
\end{document}